\documentclass[12pt]{amsart}
 
\usepackage{amsmath,amssymb,amsthm}
\usepackage[margin=3cm]{geometry}
\usepackage{enumerate}
\usepackage{amssymb,amsmath,amsthm, amscd, enumerate, mathrsfs}
\usepackage[dvips]{color}
\usepackage{version}
\usepackage[colorinlistoftodos]{todonotes}
\usepackage{url}
\usepackage{bbm}

\newtheorem{Thm}{Theorem}[section]
\newtheorem{Lem}[Thm]{Lemma}

\newtheorem{Conj}[Thm]{Conjecture}
\newtheorem{``Conj"}[Thm]{``Conjecture"}

\theoremstyle{remark}
\newtheorem{Rem}[Thm]{Remark}
\newtheorem{Ex}[Thm]{Example}
\theoremstyle{definition}
\newtheorem{Def}[Thm]{Definition}

\newtheorem{Notation}{Notation}

\newtheorem*{ack}{Acknowledgements}

\newtheorem{claim}{Claim}

\newcommand{\lct}{\mathop{\mathrm{lct}}\nolimits}

\usepackage{marginnote}

\begin{document}

\title[On arithmetic K-stability]
{Minimization of Arakelov K-energy 
for many cases}

\author{Masafumi Hattori, Yuji Odaka}
\date{\today}



\maketitle

\begin{abstract}
We prove that for 
various polarized varieties over $\overline{\mathbb{Q}}$, 
which broadly includes K-trivial case, K-ample case, Fano case, minimal models, certain classes of fibrations, 
certain metrized ``minimal-like" models minimizes 
the Arakelov theoretic analogue of the Mabuchi K-energy,  
as conjectured in \cite{OFal}. 
This is an Arakelov theoretic analogue of \cite{H22}. 
\end{abstract}

\section{Introduction}

The K-stability of polarized varieties was originally designed to 
give an algebro-geometric counterpart of the existence of canonical K\"ahler metrics \cite{Tian, Donaldson} (see \S \ref{Kst.sec} for more details). The second author introduced arithmetic framework for K-stability in \cite{OFal}, 
which discusses certain modular heights of polarized varieties  $(X,L)$ over 
$\overline{\mathbb{Q}}$, which for instance conjecturally allows 
generalization of Faltings heights of abelian varieties \cite{Fal}. The plan is to achieve it as the infimum 
or minimum of what \cite{OFal} calls 
{\it Arakelov K-energy} or {\it K-modular height} 
which depends on metrized models. 

\cite[Conjecture 3.12, 3.13]{OFal} (see our 
Conjecture \ref{aYTD}) means 
to characterize the models which attain such 
minimum, whose partial resolution is the aim of this paper. 
It is done by 
fitting the theory of ``special K-stability" by the first author \cite{H22} in usual algebraic geometry (cf., \S \ref{spKst.sec}), 
to the arithmetic framework \cite{OFal}, with some differential geometric inputs as \cite{CS, Gao,  Zhang}. 

\begin{Notation}[Arithmetic setup]\label{not.ar}
We slightly change notation from \cite{OFal} to fit more to \cite{H22}. 
Let $F$ be a number field, $X_{\eta}$ a $n$-dimensional smooth projective variety over $F$ 
and $L_{\eta}$ an ample line bundle (polarization) on it. 
We consider an ample-polarized normal projective model $(X,L)$ over $\mathcal{O}_F$, the ring of integers in $F$, 
with the generic fiber $(X_{\eta},L_{\eta})$ possibly after the extention of scalars i.e., 
replacing $F$ by its finite extension. $(X_{\mathbb{C}},L_{\mathbb{C}})$ denotes 
the base change 
$(X_{\eta},L_{\eta})\times_{F}\mathbb{C}$ 
and the reduction of $(X,L)$ over a prime ideal $\mathfrak{p}$ of $\mathcal{O}_F$ as 
$(X_{\mathfrak{p}},L_{\mathfrak{p}})$. 

We write a hermitian metric of $L_{\mathbb{C}}$ 
of real type, as 
$h_L$ and its corresponding $1$-st Chern form 
as $\omega_{h_L}$ which we assume to be positive definite. 
The pair $(L,h_L)$ is often denoted as $\overline{L}$ 
or $\overline{L}^{h_L}$. 
The dual of a line bundle is denoted by ${}^{\vee}$. 
\end{Notation}

When we focus on the complex place or 
(positive characteristic) reduction, 
we use different notations to be set 
as Notation \ref{not.C} and \ref{not.poschar} later. 

\pagebreak
The following is our main object to study. 

\begin{Def}[{\cite[\S 2]{OFal}}]\label{hK}
We define 
{\it the Arakelov-K-energy} (or {\it K-modular height}) as 
\begin{align*}
&h_{K}^{\rm Ar}(X,L,h_L)\\
&:=\frac{1}{[F:\mathbb{Q}]}
\biggl\{-\frac{n(L_{\eta}^{n-1}.K_{X_\eta})}{(n+1)(L_{\eta}^{n})}(\overline{L}^{h_L})^{n+1}
+\dfrac{((\overline{L}^{h_L})^{n}.\overline{K_{\mathcal{X}/\mathcal{O}_F}}^{\it Ric(\omega_{h_L})})}
{(L_{\eta}^{n})}
\biggr\}.
\end{align*}
In the above, the superscript 
$\it Ric(\omega_{h_L})$ means the metrization of $K_{X(\mathbb{C})}$ which corresponds 
to the Monge-Ampere measure $\omega_{h_L}^n$. 
The above is slightly different from \cite[Definition 2.4]{OFal} by a normalizing constant 
$(n+1)(L_{\eta})^n$. 
\end{Def}

We excerpt a part of the series of conjectures in \cite{OFal} as follows, which is what we partially prove 
in this paper. This is somewhat analogous to 
the CM minimization conjecture, 
introduced by the second author 
(cf., \cite{OCM1}, \cite{OCM2}), though implicitly also 
combined a little with usual Yau-Tian-Donaldson 
conjecture (Conjecture \ref{YTD}). 

\begin{Conj}[Arithmetic Yau-Tian-Donaldson conjecture \cite{OFal}]\label{aYTD} 
We fix a normal polarized projective variety $(X_{\eta},L_{\eta})$ over a 
number field $F$. 

Then, we consider all the metrized polarized normal models $(X,L,h_L)$ 
(in the sense of above Notation \ref{not.ar}) 
over $\mathcal{O}_{F'}$ 
where $F'$ also runs over all finite extensions of $F$. Then, 
$h_K(X,L,h_L)$ attains their minimum if and only if 
\begin{enumerate}
\item all the reductions $(X_{\mathfrak{p}},L_{\mathfrak{p}})$ are K-semistable, 
\item $\omega_{h_L}$ is a K\"ahler form with constant scalar curvature. 
\end{enumerate}
\end{Conj}

Recall that the attained minimum above for abelian varieties case is 
essentially the Faltings height \cite{Fal}, modulo some simple constants, 
as confirmed in \cite{OFal}. 

\vspace{2mm}
Recently, as we briefly review at the subsection 
\S \ref{spKst.sec}, 
the first author \cite{H22} introduces the notion of 
``{\it special K-stability}" 
which, nevertheless of its name, include many cases.  The notion is defined by using J-stability (cf., \S \ref{Jst.sec}, 
\cite{H21}) 
and the  
$\delta$-invariant (cf., \S \ref{delta.sec}, 
\cite{FO, BJ}) 
in the field of K-stability. 
Then, the first author showed the special K-(semi)stability 
implies the usual K-(semi)stability \cite{H22a}
(see \S \ref{Kst.sec} for more details). 

\vspace{2mm}
Our main theorem \ref{hK.min} is 
roughly as 
follows, which partially confirms the ``if" direction 
of the above Conjecture \ref{aYTD}. 
\begin{Thm}[{Main Theorem ($=$Theorem \ref{hK.min})}]\label{hK.min.intro}
If a metrized polarized model $(X,L,h_L)/\mathcal{O}_F$ 
satisfies analogues of special K-stability 
over any place of $F$, 
$h_K(X,L,h_L)$ attains the minimum 
for the fixed geometric generic fiber 
$(X,L)\times_{F}\overline{\mathbb{Q}}$ over 
$\overline{\mathbb{Q}}$. 
\end{Thm}
There are many classes of 
polarized varieties which 
have a model which satisfies the 
above-mentioned condition, 
such as 
K-trivial case, K-ample case, K-stable Fano varieties case, minimal models and 
some fibrations for instance. 
Thus, the above theorem at least broadly 
generalizes \cite[Theorem 3.14]{OFal}. 
For instance, 
the cases of 
minimal models and 
certain 
algebraic 
fibrations 
are newly included (compare  \cite{H21}). 
Large parts of this paper are devoted 
to preparations to 
rigorously formulate (and prove) Theorem \ref{hK.min.intro}, finally 
resulting to the main theorem \ref{hK.min}. Some 
review of the background is also contained in the next section \S \ref{prelim} 
for the readers convenience. 

Our discussion of the K-modular height 
(Definition \ref{hK}) 
and its variant is based on the framework of \cite{OFal}, which uses the Gillet-Soul\'e 
intersection theory \cite{GS} and its developments. The main discussion of 
the proof of Theorem \ref{hK.min} closely follow 
\cite{H22}. Thus, we also refer to them 
for more details of the background. 

\section{Preliminaries}\label{prelim}

This section consists of three subsections. 
The first \S \ref{Kst.sec} 
briefly reviews some basics of K-stability for convenience of readers, introducing the J-stability and the 
$\delta$-invariant as well. 
They are both recent useful tools to study K-stability. 
The latter two subsections 
\S \ref{twist.sec}, \S \ref{Frob.sec} 
are technical results prepared 
for our main theorem in the 
next section \S \ref{MThm.sec}. 
The materials from the section 
\ref{Frob.sec} and later are new. 

\subsection{Review of K-stability}\label{Kst.sec}

In this subsection, we review the usual K-stability in the 
complex geometric setup. 
For that, we first re-set the 
notation for this subsection \S \ref{Kst.sec}: 

\begin{Notation}[Complex geometric setup]\label{not.C}
We consider a polarized smooth projective 
variety 
$(X_{\mathbb{C}},L_{\mathbb{C}})$ 
over $\mathbb{C}$ 
which, in this subsection, 
does not necessary descends over 
$\overline{\mathbb{Q}}$ as Notation \ref{not.ar}. 
As in Notation \ref{not.ar}, we denote 
a hermitian metric of $L_{\mathbb{C}}$, as 
$h_L$ and its corresponding $1$-st Chern form as $\omega_{h_L}$ which we assume to be positive definite i.e., a K\"ahler form. 
For a smooth real function $\varphi$ 
on $X(\mathbb{C})$, we set 
$\omega_{\varphi}
:=\omega_{h_L}
+\sqrt{-1}
\partial\overline{\partial}\varphi. 
$

\end{Notation}

\subsubsection{K-stability $($\cite{Tian,Donaldson}$)$}
Now, we review the definition of 
the K-stability. 
\begin{Def}
A {\textit{test configuration}} of $(X_{\mathbb{C}},L_{\mathbb{C}})$ 
of the exponent $r(\in \mathbb{Z}_{>0})$ means 
a projective scheme $\mathcal{X}_{\mathbb{C}}$ 
flat over $\mathbb{P}^1_{\mathbb{C}}$, a relatively ample line bundle $\mathcal{L}_{\mathbb{C}}$ 
on $\mathcal{X}_{\mathbb{C}}$, $\mathbb{G}_m$-action on 
$(\mathcal{X}_{\mathbb{C}},\mathcal{L}_{\mathbb{C}})$ 
together with a $\mathbb{G}_m$-equivariant 
isomorphism 
\begin{equation*}
(\mathcal{X}_{\mathbb{C}},\mathcal{L}_{\mathbb{C}})|_{(\mathbb{P}^1_{\mathbb{C}}\setminus \{0\})}
\simeq (X_{\mathbb{C}},
L_{\mathbb{C}}^{\otimes r})\times (\mathbb{P}^1_{\mathbb{C}}\setminus \{0\}). 
\end{equation*}
We simply denote this set of data, 
forming a test configuration, as 
$(\mathcal{X}_{\mathbb{C}},\mathcal{L}_{\mathbb{C}})$. 
\end{Def}

\begin{Def}[\cite{Donaldson, Wang, ODF}]
\label{DF.def}
The 
{\textit{Donaldson-Futaki invariant}} 
${\rm DF}(\mathcal{X}_{\mathbb{C}},\mathcal{L}_{\mathbb{C}})$
of a test configuration 
$(\mathcal{X}_{\mathbb{C}},\mathcal{L}_{\mathbb{C}})$, where $\mathcal{X}_{\mathbb{C}}$ is normal, 
is defined as 
$$
\dfrac{-n(L_{\mathbb{C}}^{n-1}.
K_{X_{\mathbb{C}}})}{(n+1)(L_{\mathbb{C}})^n}
(\mathcal{L}_{\mathbb{C}})^{n+1}
+
r(\mathcal{L}_{\mathbb{C}}^{n}.
K_{\mathcal{X}_{\mathbb{C}}/\mathbb{P}^1_{\mathbb{C}}}). 
$$
Note that, thanks to 
the homogeneity of the above, 
it is convenient to 
replace $\mathcal{L}$ by 
$\mathcal{L}/r$ as a $\mathbb{Q}$
-line bundle (of exponent $1$). 
We say 
$(X_{\mathbb{C}},L_{\mathbb{C}})$ 
is {\textit{K-stable}} 
(resp., {\textit{K-semistable}}) 
if 
they are positive unless 
$\mathcal{X}_{\mathbb{C}}$ 
is $X_{\mathbb{C}}\times 
\mathbb{P}^1$ 
(resp., 
they are always non-negative). 
We also say 
$(X_{\mathbb{C}},L_{\mathbb{C}})$ 
is {\textit{K-polystable}} 
if 
they are positive unless 
$\mathcal{X}_{\mathbb{C}}$ 
is a $X_{\mathbb{C}}$-fiber bundle over 
$\mathbb{P}^1$. 
\end{Def}
The Donaldson-Futaki invariant 
is recently also called 
{\it non-archimedean Mabuchi energy} (cf., \cite{BHJ}) 
modulo a technical slight difference. 
Also, note that our definition of the K-modular height 
(Definition \ref{hK}) is designed after the above 
intersection number formula. 
The original motivation for K-stability is 
the following well-known 
conjecture in complex geometry. 

\begin{Conj}[The Yau-Tian-Donaldson conjecture \cite{Donaldson}]
\label{YTD}
For any polarized smooth complex 
projective 
variety 
$(X_{\mathbb{C}},L_{\mathbb{C}})$, 
it 
is K-polystable if and only if 
$X_{\mathbb{C}}$ admits a 
constant scalar curvature K\"ahler metric 
of the K\"{a}hler class 
$c_1(L_{\mathbb{C}})$. 
\end{Conj}

\subsubsection{J-stability}\label{Jst.sec}
The J-stability of polarized variety is a 
certain toy-model analogue of the K-stability, 
originally named after the J-flow of Donaldson \cite{Don99} (cf., also 
\cite{Chen}). 
After Notation \ref{not.C}, 
we further consider 
another auxiliary ample $\mathbb{R}$-line bundle 
$H$ of $X_{\mathbb{C}}$. 

The differential geometric counterpart 
of the J-stability (Definition 
\ref{Jst.def}) 
is the following so-called 
{\it J$^\chi$-equation} 
\begin{equation}
\label{Jeqnn}
\mathrm{tr}_{\omega}(\chi)= \text{constant},
\end{equation}
where ${\rm tr}_{\omega}$ means 
the trace with respect to $\omega$. 
See e.g., \cite{LS, Gao, 
DatarPingali, So, H21} for more detailed context. Here, we only briefly review it at the level we use in this paper.  

\begin{Def}[J-stability]\label{Jst.def}
For a test configuration 
$(\mathcal{X}_{\mathbb{C}},\mathcal{L}_{\mathbb{C}})$
of a polarized variety $(X_{\mathbb{C}},L_{\mathbb{C}})$, 
we take a resolution of 
indeterminancy of birational map 
$X_{\mathbb{C}}
\times \mathbb{P}^1
\dashrightarrow 
\mathcal{X}_{\mathbb{C}}$ 
as 
$$X_{\mathbb{C}}
\times \mathbb{P}^1
\xleftarrow{p}\mathcal{Y}
\xrightarrow{q}
\mathcal{X}_{\mathbb{C}},$$ 
so that $p$ and $q$ are morphisms. 
We also denote the 
first projection 
$X_{\mathbb{C}}
\times \mathbb{P}^1\to 
X_{\mathbb{C}}$ as $p_1$. 
Then 
we define 
$$\mathcal{J}^{H,\mathrm{NA}}
(\mathcal{X}_{\mathbb{C}},\mathcal{L}_{\mathbb{C}})
:=
\dfrac{-n(L_{\mathbb{C}}^{n-1}.
H)}{(n+1)(L_{\mathbb{C}})^n}
(\mathcal{L}_{\mathbb{C}})^{n+1}
+
r(\mathcal{L}_{\mathbb{C}}^{n}.
(p_1\circ p)^* H). 
$$

A polarized variety $(X_{\mathbb{C}},L_{\mathbb{C}})$ is 
$\mathrm{J}^H${\textit{-semistable}} 
if $\mathcal{J}^{H,{\rm NA}}
(\mathcal{X}_{\mathbb{C}},\mathcal{L}_{\mathbb{C}})
\ge 0$ for any test configuration.
$(X_{\mathbb{C}},L_{\mathbb{C}})$ is called {\it uniformly} $\mathrm{J}^H${\textit{-stable}} if there exists $\epsilon>0$ such that $(X_{\mathbb{C}},L_{\mathbb{C}})$ is $\mathrm{J}^{H-\epsilon L_{\mathbb{C}}}$-semistable.

\end{Def}

\begin{Rem}\label{rem--j-for-positive}

Note that the above Definition \ref{Jst.def} 
does not particularly use that the 
base field is $\mathbb{C}$. Hence, 
    we can also define J-stability of polarized varieties over any field, including positive characteristic, in the same way. 
\end{Rem}

The analogue of Yau-Tian-Donaldson conjecture 
for the J-stability is now a theorem, as conjectured by 
Lejmi-Szekelyhidi \cite{LS}. 

\begin{Thm}[\cite{Gao, DatarPingali, So}] 
Fix a K\"ahler form $\chi$ such that 
$[\chi]=c_{1}(H_{\mathbb{C}})$. Then, 
the following are equivalent:
\begin{enumerate}
\item There is a (unique) K\"ahler form 
$\omega$ such that $[\omega]=\mathrm{c}_1(L_{\mathbb{C}})$ which 
 satisfies the J-equation \eqref{Jeqnn} above. 
\item $(X_{\mathbb{C}},
L_{\mathbb{C}})$ 
is uniformly 
$J^H$-stable. 
\end{enumerate}
\end{Thm}
Here, uniform J-stability above is a slight strengthening of the J-stability i.e., it implies J-stability, after the idea of 
\cite{BHJ, Dervan}. 

Next, we recall the definition of filtrations.

\begin{Def}\label{def--j-filt}
    Let $X$ be a proper reduced scheme over a field $\mathbbm{k}$ with an ample line bundle $L$.
    Suppose that $H^0(X,L^{\otimes m})$ is generated by $H^0(X,L)$ for any $m\in \mathbb{Z}_{>0}$. 
    We call a set of subspaces $\mathscr{F}:=\{\mathscr{F}^{\lambda} H^0(X,L^{\otimes m})\}_{m\in\mathbb{Z}_{>0},\lambda\in\mathbb{Z}}$ of $H^0(X,L^{\otimes m})$ a {\it filtration} if $\mathscr{F}$ satisfies the following:
\begin{enumerate}  
\item $\mathscr{F}^{\lambda} H^0(X,L^{\otimes m})\cdot \mathscr{F}^{\lambda'} H^0(X,L^{\otimes m'})\subset\mathscr{F}^{\lambda+\lambda'} H^0(X,L^{\otimes m+m'})$,
\item $\mathscr{F}^{\lambda} H^0(X,L^{\otimes m})\subset\mathscr{F}^{\lambda'} H^0(X,L^{\otimes m})$ for any $\lambda>\lambda'$,
\item there exists $N>0$ such that if $\lambda>Nm$ then $\mathscr{F}^{\lambda} H^0(X,L^{\otimes m})=0$ and if $\lambda\le Nm$ then $\mathscr{F}^{\lambda} H^0(X,L^{\otimes m})=H^0(X,L^{\otimes m})$.
 \end{enumerate}
 For any $\mathscr{F}$, we fix $N$ as above.
 Then we take an approximation to $\mathscr{F}$ as follows.
 For any $l\in\mathbb{Z}_{>0}$, let $\mathscr{F}_{(l)}$ be a filtration generated by $\mathscr{F}^{t}H^0(X,L^{\otimes m})$ and $\mathscr{F}^{\lambda}H^0(X,L^{\otimes l})$ for any $t\le -Nm$, $\lambda$ and $m$ as \cite[Definition 2.16]{H22}.
 We call a sequence $\{\mathscr{F}_{(l)}\}_{l\in\mathbb{Z}_{>0}}$ an {\it approximation} to $\mathscr{F}$.
 We take the normal test configuration $(\mathcal{X}^{(l)},\mathcal{L}^{(l)})$ for $(X,L)$ 
 induced by $\mathscr{F}_{(l)}$ as \cite[Definition 2.19]{H22}, which is defined as follows.
 Let $\mathfrak{a}_{(l)}$ be the image of the following $\mathcal{O}_X[t]$-homomorphism
 \begin{equation}
 \bigoplus_{\lambda\in\mathbb{Z}} t^{-\lambda}\mathscr{F}^\lambda H^0(X,L^{\otimes l})\otimes\mathcal{O}_{X\times\mathbb{A}^1}(-l(L\times\mathbb{A}^1))\to \mathcal{O}_{X}[t,t^{-1}],\nonumber 
 \end{equation}
 where $t$ is the canonical coordinate of $\mathbb{A}^1$.
 Then, $\mu_l\colon \mathcal{X}^{(l)}\to X\times \mathbb{P}^1$ be the blow up along $\mathfrak{a}_{(l)}$ and $$\mathcal{L}^{(l)}:=\mu_l^*(L\times\mathbb{P}^1)-\frac{1}{l}\mu_l^{-1}(\mathfrak{a}_{(l)}).$$
 Let $H$ be an ample divisor on $X$ and take $D\in |mH|$ for any $m\in\mathbb{Z}_{>0}$.
 We say that $D$ is {\it compatible} with $\{\mathscr{F}_{(l)}\}$ if the support of $\mu_l^*D\times\mathbb{P}^1$ contains no $\mu_l$-exceptional divisor for any $l$.
 Finally, we close this subsection with the following lemma.
 \begin{Lem}[{\cite[Lemma 2.20]{H22}}]\label{lem-j-filt}
 In the above situation, we have that
 \[
     \mathcal{J}^{H,\mathrm{NA}}(\mathscr{F}):=\lim_{l\to\infty}\mathcal{J}^{H,\mathrm{NA}}(\mathcal{X}^{(l)},\mathcal{L}^{(l)}).
     \]
     If $(X,L)$ is further $\mathrm{J}^H$-semistable, then 
     \[
    \mathcal{J}^{H,\mathrm{NA}}(\mathscr{F})\ge0.
     \]
 \end{Lem}
 \begin{proof}
 If $\mathbbm{k}$ is uncountable, then we can choose a compatible divisor $D\in |mH|$ for some $m$ with $\{\mathscr{F}_{(l)}\}$.
     Thus, $\lim_{l\to\infty}\mathcal{J}^{H,\mathrm{NA}}(\mathcal{X}^{(l)},\mathcal{L}^{(l)})$ exists and coincides with the value \cite[(5)]{H22} by {\cite[Lemma 2.20]{H22}} (whose proof also works for the positive characteristic case).  
     For the general case, we reduce to the previous case by changing the base field $\mathbbm{k}$ to some uncountable field (cf.~\cite[Remark 2.21]{H22}).
 \end{proof}
\end{Def}

\subsubsection{Delta invariant $($\cite{FO, BJ}$)$}
\label{delta.sec}

First, we recall the log canonical threshold.
\begin{Def}
    Let $X_{\mathbb{C}}$ be a normal variety over $\mathbb{C}$ such that $K_{X_{\mathbb{C}}}$ is $\mathbb{Q}$-Cartier with an effective $\mathbb{Q}$-Cartier $\mathbb{Q}$-Weil divisor $D$ on $X_{\mathbb{C}}$. 
For any prime divsior $E$ over $X_{\mathbb{C}}$, we set the log discrepancy 
\[
A_{X_{\mathbb{C}}}(E)=1+\mathrm{ord}_E(K_Y-\pi^*K_X),
\]
where $\pi\colon Y\to X$ is a log resolution such that $E$ is a divisor on $Y$.
Then, we set the {\it log canonical threshold} of $X_{\mathbb{C}}$ with respect to $D$ as
\[
\inf_{E}\frac{A_{X_{\mathbb{C}}}(E)}{\mathrm{ord}_E(D)},
\]
where $E$ runs over all prime divisors over $X_{\mathbb{C}}$.
\end{Def}

The {\textit{delta invariant}} $\delta(X_{\mathbb{C}},L_{\mathbb{C}})$
(a.k.a. 
{\textit{stability thresholds}})
introduced by \cite{FO, BJ} 
is a real invariant of a Fano variety or a 
general polarized variety $(X_{\mathbb{C}},L_{\mathbb{C}})$ 
(see Notation \ref{not.C}) 
which is now known to give an 
effective criterion 
for the K-stability. 
For its definition, 
a special type of $\mathbb{Q}$-divisors 
as follows 
is introduced (\cite{FO}). 

\begin{Def}[{\cite{FO}}]
Let $k$ be a natural number.  
For any basis 
\[
s_{1},\dots,s_{h^{0}(L_{\mathbb{C}}^{\otimes k})}\] 
of $H^{0}(L_{\mathbb{C}}^{\otimes k})$, taking the corresponding 
divisors $D_{1},\dots,D_{h^{0}(L_{\mathbb{C}}^{\otimes k})}$ (note $D_{i}\sim L_{\mathbb{C}}^{\otimes k}$), 
we obtain a $\mathbb{Q}$-divisor 
\[
D:=\frac{D_{1}+\cdots+D_{h^{0}(L_{\mathbb{C}}^{\otimes k})}}{k\cdot h^{0}(L_{\mathbb{C}}^{\otimes k})}.
\] 
This kind of effective 
$\mathbb{Q}$-divisor is called an 
\textit{$(\mathbb{Q}$-$)$divisor of $k$-basis type}. 
\end{Def}

\begin{Def}[\cite{FO}]
For $k\in\mathbb{Z}_{>0}$, we define 
\[
\delta_{k}(X_{\mathbb{C}},L_{\mathbb{C}}):=\inf_{
\substack{
(L_{\mathbb{C}}\sim_\mathbb{Q}) D;\\
D:\,\,k\text{-basis type}
}} \lct(X_{\mathbb{C}};D), 
\]
where lct stands for 
log canonical thresholds.
It is easy to see that there exist a prime divisor $E$ over $X_{\mathbb{C}}$ and a divisor $D$ of $k$-basis type such that
\[
\delta_{k}(X_{\mathbb{C}},L_{\mathbb{C}})=\frac{A_{X_{\mathbb{C}}}(E)}{\mathrm{ord}_E(D)}.
\]
Then, we say that $E$ {\it computes} $\delta_{k}(X_{\mathbb{C}},L_{\mathbb{C}})$.
On the other hand, we set 
\[
\delta(X_{\mathbb{C}},L_{\mathbb{C}}):=\lim_{k\to \infty}\delta_{k}(X_{\mathbb{C}},L_{\mathbb{C}}). 
\]
The above limit is known to exists 
by \cite{BJ}. 
Its original motivation is the following criterion. 
\end{Def}
\begin{Thm}[\cite{FO, BJ}]
For any Fano manifold 
$X_{\mathbb{C}}$,
$\delta(X_{\mathbb{C}},-K_{X_{\mathbb{C}}})> 1$ 
(resp., $\ge 1)$ 
then $(X_{\mathbb{C}},-K_{X_{\mathbb{C}}})$ is 
uniformly K-stable (resp., K-semistable). 
\end{Thm}
Here, again, the uniform K-stability above is a priori strengthening of the K-stability due to \cite{BHJ, Dervan} but 
more recently 
they are 
confirmed to be equivalent again for anticanonically polarized 
$\mathbb{Q}$-Fano varieties (\cite{LXZ}). 
\begin{Rem}
    We note that $\delta_k(X_K,L_K)$ for any $k\in\mathbb{Z}_{>0}$ and $\delta(X_K,L_K)$ for any polarized klt pair over any field in the same way as the complex field. See also \cite[Definition 2.3]{Zhu}.\label{rem-delta}
\end{Rem}

\subsubsection{Special K-stability $($\cite{H22a, H22}$)$}
\label{spKst.sec}

Recently the delta invariant turned 
out to be also efficient for 
studying K-stability of more general 
varieties (cf., e.g., 
\cite{Zhang, H22a, H22}). 
In particular, \cite{H22}  
introduces the following notion which also forms a 
key idea of the current paper. 

\begin{Def}[{\cite[Definition 3.10]{H22}}]\label{specialKss}
We call a polarized complex variety $(X_{\mathbb{C}},L_{\mathbb{C}})$ 
is {\it specially K-stable} (resp.~{\it specially K-semistable}) 
if $X_{\mathbb{C}}$ is semi-log-canonical and 
both of the following hold: 
\begin{enumerate}
\item $K_{X_{\mathbb{C}}}+\delta(X_{\mathbb{C}},L_{\mathbb{C}})
L_{\mathbb{C}}$ is ample (resp.~nef), 
\item $(X_{\mathbb{C}},L_{\mathbb{C}})$ is uniformly $\mathrm{J}^{K_{X_{\mathbb{C}}}+\delta(X_{\mathbb{C}},L_{\mathbb{C}})
L_{\mathbb{C}}}$-stable 
(resp.~$\mathrm{J}^{K_{X_{\mathbb{C}}}+\delta(X_{\mathbb{C}},L_{\mathbb{C}})
L_{\mathbb{C}}}$-semistable). 
\end{enumerate}
\end{Def}

The main point of this notion is: 

\begin{Thm}[{\cite[Corollary 3.21]{H22}}]\label{spKss.Kss}
If $(X_{\mathbb{C}},L_{\mathbb{C}})$ is specially K-stable (resp.~specially K-semistable), 
it is also uniformly K-stable (resp.~specially K-semistable). 
\end{Thm}

As \cite[Theorem 3.12]{H22} reviews (cf., also e.g., \cite{H21}), 
there are many classes of 
polarized varieties which satisfy special K-(semi)stability. 

\subsection{Positive characteristic analogue of 
$\delta$-invariant}\label{Frob.sec}

Now we turn to a preparation for the reductions at non-archimedean places, 
which is to introduce a positive characteristic analogue of the $\delta$-invariant (\cite{FO, BJ}). In the next section, we  use it to 
formulate a positive characterisitics analogue of 
special K-semistability (Definition \ref{specialKss}). 

Since the following arguments work more 
generally, i.e., not only for reductions of 
arithmetic models, we use the following 
(compatible) notation in this subsection 
\S \ref{not.poschar}: 

\begin{Notation}[Positive characteristic setup]\label{not.poschar}

$X_{\mathfrak{p}}$ 
is a projective scheme over a field of positive characteristic, and $L_{\mathfrak{p}}$ is an ample line bundle on it. 
Unlike Notation \ref{not.ar}, 
$(X_{\mathfrak{p}},L_{\mathfrak{p}})$ 
does not necessarily lift to 
$\mathcal{O}_{\overline{\mathbb{Q}}}$, 
the ring of integers in $\overline{\mathbb{Q}}$. 
\end{Notation}

\begin{Def}[Frobenius $\delta$-invariant]\label{Frob.delta}
For a triple $(X_{\mathfrak{p}},\Delta,L_{\mathfrak{p}})$ of geometrically normal projective variety $X_{\mathfrak{p}}$ over a field of characteristic $p>0$, 
its effective $\mathbb{Q}$-Weil divisor such that $K_{X_{\mathfrak{p}}}+\Delta$ is $\mathbb{Q}$-Cartier, 
$(X_{\mathfrak{p}},\Delta)$ being locally F-pure (\cite[Definition 2.1]{HW}) 
an ample line bundle $L_{\mathfrak{p}}$ over $X_{\mathfrak{p}}$, we consider the following invariants. 
\begin{enumerate}
\item 
For a positive integer $k$, 
we set {\it the $k$-(quantized) Frobenius $\delta$-invariant} 
$$\delta_{(X_{\mathfrak{p}},\Delta),k}^F(L_{\mathfrak{p}}):=\inf_{
\substack{
(L_{\mathfrak{p}} \sim_\mathbb{Q}) D;\\
D:\,\,k\text{-basis type}
}} {\rm Fpt}((X_{\mathfrak{p}},\Delta);D),
$$
where ${\rm Fpt}$ denotes the F-pure threshold 
$$\sup\{c\mid (X_{\mathfrak{p}},\Delta+cD) \text{ is locally F-pure } \}$$
as originally introduced in \cite[\S 2]{TW} for the affine setup. Here, $D$ runs over all $k$-basis type divisors for $L_{\mathfrak{p}}$ in the sense of \cite[Definition 0.1]{FO}, \cite[Introduction]{BJ}. 
\item 
Then we define {\it the Frobenius $\delta$-invariant} as 
$$\delta_{(X_{\mathfrak{p}},\Delta)}^F(L_{\mathfrak{p}}):=\liminf_{k\to \infty}\delta_{(X_{\mathfrak{p}},\Delta),k}^F(L_{\mathfrak{p}}).$$ 
\end{enumerate}
\end{Def}
Recall that if we replace ${\rm Fpt}$ by ${\rm lct}$, the above is nothing but 
$\delta_{(X_{\mathfrak{p}},\Delta)}(L_{\mathfrak{p}})$ in the original 
form \cite{FO} (see also \cite{BJ}). 
We sometimes omit $(X_{\mathfrak{p}},\Delta)$ from the subscripts in the above and simply write $\delta_k^F(L_{\mathfrak{p}})$ and $\delta^F(L_{\mathfrak{p}})$ respectively. 
On the other hand, we can define $\delta^F_{(\overline{X_{\mathfrak{p}}},\overline{\Delta},k)}(\overline{L_{\mathfrak{p}}})$ and $\delta^F_{(\overline{X_{\mathfrak{p}}},\overline{\Delta})}(\overline{L_{\mathfrak{p}}})$ in the same way, i.e.~
\begin{align*}
    \delta_{(\overline{X_{\mathfrak{p}}},\overline{\Delta}),k}^F(\overline{L_{\mathfrak{p}}})&:=\inf_{
\substack{
(\overline{L_{\mathfrak{p}}} \sim_\mathbb{Q}) D;\\
D:\,\,k\text{-basis type}
}} {\rm Fpt}((\overline{X_{\mathfrak{p}}},\overline{\Delta});D),\\
\delta_{(\overline{X_{\mathfrak{p}}},\overline{\Delta})}^F(\overline{L_{\mathfrak{p}}})&:=\liminf_{k\to \infty}\delta_{(\overline{X_{\mathfrak{p}}},\overline{\Delta}),k}^F(\overline{L_{\mathfrak{p}}}),
\end{align*} where $(\overline{X_{\mathfrak{p}}},\overline{\Delta},\overline{L_{\mathfrak{p}}})=(X_{\mathfrak{p}},\Delta,L_{\mathfrak{p}})\times_{\mathcal{O}_F/\mathfrak{p}}\overline{\mathcal{O}_F/\mathfrak{p}}$ and $\overline{\mathcal{O}_F/\mathfrak{p}}$ denotes the algebraic closure of $\mathcal{O}_F/\mathfrak{p}$.
We sometimes simply write them as $\delta_k^F(\overline{L_{\mathfrak{p}}})$ and $\delta^F(\overline{L_{\mathfrak{p}}})$ respectively. 
We note that 
\begin{align*}
    \delta_k^F(\overline{L_{\mathfrak{p}}})&\le \delta_k^F(L_{\mathfrak{p}})\\
    \delta^F(\overline{L_{\mathfrak{p}}})&\le \delta^F(L_{\mathfrak{p}}).
\end{align*}

We remark that there are some examples 
where the above inequality can be strict 
(indeed, similar arguments to Remark \ref{rem:deltacompare} 
applies to nontrivial twists of elliptic curves). 
Also note that by the simple combination of \cite[3.3]{HW}, \cite[2.2(5)]{TW} and \cite[A]{BJ}, we have 
$\delta_{(X_{\mathfrak{p}},\Delta)}(L)\ge \delta_{(X_{\mathfrak{p}},\Delta)}^F(L_{\mathfrak{p}}).$ Note also that the above definition naturally extends to $cL$ with a line bundle $L$ and $c\in \mathbb{R}_{>0}$ as $\delta_{(X_{\mathfrak{p}},\Delta),k}^F(cL_{\mathfrak{p}})=\frac{1}{c}\delta_{(X_{\mathfrak{p}},\Delta),k}^F(L_{\mathfrak{p}})$.

\subsection{Twisted analogue of 
\cite{OFal}}\label{twist.sec}
We now go back to the Arakelov geometric setup, and discuss 
after Notation \ref{not.ar} henceforth. 
The following are auxiliary ``twisted" analogues of the original Arakelov K-energy 
(Definition \ref{hK}) of \cite{OFal} and 
its variants. 
The ``twist" here refers to 
consideration of (again) an additional 
hermitian-metrized line bundle $(H,h)$ 
so that the original untwisted setup 
means the case when $H=\mathcal{O}_X$ 
and $h$ is trivial metric over any 
infinite place. See e.g., \cite{Dervan} for more 
background. 

We refrain from considering any boundary, 
i.e., ``logarithmic" extension with mild singularities, to avoid non-substantial technical complications. 
We use these to partially prove Conjecture \ref{aYTD}, resulting to our main Theorem \ref{hK.min}. 

\begin{Def}
\label{A-functionals}
Fix an ample-polarized normal projective model $(X,L)$ over $\mathcal{O}_F$ with $h_L$ as Notation \ref{not.ar}.
\begin{enumerate}
\item 
For a metrized line bundle $\overline{H}=(H,h)$ on $X$, we define 
{\it the $(H,h)$-twisted Arakelov-K-energy} as 
\begin{align*}
&h_{K,\bar{H}}^{\rm Ar}(X,L,h_L)\\
&:=\frac{1}{[F:\mathbb{Q}]}
\biggl\{-\frac{n(L_{\eta}^{n-1}.K_{X_\eta}\otimes H|_{X_{\eta}})}{(n+1)(L_{\eta}^{n})^2}(\overline{L}^{h_L})^{n+1}
+\dfrac{((\overline{L}^{h_L})^{n}.\overline{K_{\mathcal{X}/\mathcal{O}_F}}^{\it Ric(\omega_{h_L})}\otimes \overline{H}^h)}
{(L_{\eta}^{n})}
\biggr\}.
\end{align*}
In the above, the intersection numbers of the metrized 
line bundles on the total spaces 
are that of \cite{GS} and 
the superscript 
$\it Ric(\omega_{h_L})$ means the metrization of $K_{X(\mathbb{C})}$ which corresponds 
to the Monge-Ampere measure $\omega_{h_L}^n$. 
Note that if $(H,h)$ is trivial, the above quantity is nothing but Definition \ref{hK}. 
\cite[Definition 2.4]{OFal} modulo a normalizing constant 
$(n+1)(L_{\eta})^n$. 
\item 
For a line bundle $H$ on $X$, with a real type hermitian metric $h$ on $H|_{X_\eta}(\mathbb{C})$, 
we suppose $L_{\eta}=(K_{X_\eta}\otimes H|_{X_{\eta}})^{\vee}$. Then, we define 
{\it the $(H,h)$-twisted Arakelov-Ding functional} as 
$$\mathcal{D}_{\bar{H}}^{\rm Ar}(X,L,h_L):=\frac{1}{[F:\mathbb{Q}]}
\biggl\{
-\dfrac{(\overline{L})^{n+1}}{(n+1)(L_{\eta})^n}
+\widehat{\rm deg}H^0(K_{X/\mathcal{O}_F}\otimes \overline{H}^{h} \otimes \overline{L}^{h_L})
\biggr\},
$$
\noindent
where $H^0(K_{X/\mathcal{O}_F}\otimes \overline{H}^{h} \otimes \overline{L}^{h_L})$ 
is associated with the $L^2$-metric. 
\item For a metrized line bundle $\overline{H}=(H,h)$ on $X$, we set the {\it Arakelov-$\mathrm{J}^{\bar{H}}$-energy} as
\begin{align*}
\mathcal{J}^{{\rm Ar},\bar{H}}(X,L,h_L):=\frac{1}{[F:\mathbb{Q}]}
\biggl\{-\frac{n(L_{\eta}^{n-1}.H|_{X_{\eta}})}{(n+1)(L_{\eta}^{n})^2}(\overline{L}^{h_L})^{n+1}
+\dfrac{((\overline{L}^{h_L})^{n}.\overline{H}^h)}
{(L_{\eta}^{n})}
\biggr\}.
\end{align*}
We note that  $h_{K,\bar{H}}^{\rm Ar}(X,L,h_L)=h_{K}^{\rm Ar}(X,L,h_L)+\mathcal{J}^{{\rm Ar},\bar{H}}(X,L,h_L)$ holds.
\end{enumerate}
\end{Def}

\begin{Lem}
\label{2.12.analogue}
If $L_{\eta}=(K_{X_\eta}\otimes H|_{X_{\eta}})^{\vee}$, 
then $$h_{K,\bar{H}^h}^{\rm Ar}(X,L,h_L)\ge \mathcal{D}_{\bar{H}^h}^{\rm Ar}(X,L,h_L)-(L_{\eta}^n)\log(L_{\eta}^n),$$
so that 
\begin{align}\label{hKdecomp}
h_{K}^{\rm Ar}(X,L,h_L)\ge 
\mathcal{J}^{{\rm Ar},-\bar{H}^h}(X,L,h_L)+
\mathcal{D}_{\bar{H}^h}^{\rm Ar}(X,L,h_L)-(L_{\eta}^n)\log(L_{\eta}^n). 
\end{align}
Furthermore, equality holds if $L=(K_X\otimes H)^{\vee}$ and $\omega_{h_L}$ is the $\omega_h$-twisted K\"ahler-Einstein metric, 
where $\omega_h$ is the curvature form of $h$. 
\end{Lem}
The non-twisted version is discussed in \cite[Prop 7.3]{AB}, which we generalize here. 
\begin{proof}
Since $(L\otimes H)|_{X_{\eta}}=-K_{X_{\eta}}$, 
its hermitian metric $h_{L}\cdot h$ determines a (non-holomorphic) volume form $\nu$ on $X(\mathbb{C})$. 
Then, 
\begin{align}
&h_{K,\bar{H}}^{\rm Ar}(X,L,h_L)- \mathcal{D}_{\bar{H}}^{\rm Ar}(X,L,h_L)\nonumber\\
=&\frac{1}{(n+1)(L_{\eta})^n}((\overline{L}^{h_L})^n.\overline{L}^{h_L}\otimes \overline{K_{X/\mathcal{O}_F}}^{{\rm Ric}(\omega_{h_L})})\nonumber\\
-&\widehat{{\rm deg}}H^0(X,\overline{L}^{h_L}\otimes \overline{K_{X/\mathcal{O}_F}}^{{\rm Ric}(\omega_{h_L})}\otimes \overline{H}^{h}), \label{KD3}
\end{align}
where $H^0(X,\overline{L}^{h_L}\otimes \overline{K_{X/\mathcal{O}_F}}^{{\rm Ric}(\omega_{h_L})}\otimes \overline{H}^{h})$ is regarded as 
a $\mathcal{O}_F$-module with the $L^2$-metric. If we take a section $s$ of $L\otimes H\otimes K_{X/\mathcal{O}_F}$ 
which is non-vanishing at the generic fiber, it decides an effective vertical divisor $D={\rm div}(s)$, 
which we further decompose as ${\rm div}(D_F)+D'$ where $D_F$ is a divisor of $\mathcal{O}_F$ and 
$D'$ is still effective which does not contain any non-trivial (scheme-theoretic) fiber. 
Note that the weight of our metric on $L(\mathbb{C})\otimes H(\mathbb{C})\otimes K_{X(\mathbb{C})}$ is 
$\log\dfrac{\omega_{h_L}^n}{\nu}$. Hence, we continue the standard calculation as 
\begin{align*}
\eqref{KD3}&=\frac{1}{(n+1)(L_{\eta})^n}(L^n.D') (\ge 0)+\int_{X(\mathbb{C})}\log\Big(\dfrac{\omega_{h_L}^n}{\nu}\Big)  \omega_{h_L}^n \\
			&\ge \int_{X(\mathbb{C})}\log (L_{\eta}^n) \omega_{h_L}^n(=(L_{\eta}^n)\log (L_{\eta}^n))+
			\int_{X(\mathbb{C})}\log\dfrac{(\omega_{h_L}^n)/(L_{\eta}^n)}{\nu}\omega_{h_L}^n.\\ 
\end{align*}
We finally apply the Jensen's inequality for the logarithmic function to the last relative entropy term to finish the proof. 
\end{proof}

\section{Main theorem and the proof}\label{MThm.sec}

Now, we are ready to state and prove our 
main theorem as follows. It partially proves  ``if direction" of Conjecture \ref{aYTD} 
 for the case of {\it specially K-stable varieties} in the sense of \cite{H22}. 
A point is that, nevertheless of the adjective ``special", it broadly includes many cases, hence 
in particular 
generalizing the  results of \cite[Theorem 3.14]{OFal}. 

\begin{Thm}[Main theorem]\label{hK.min}
Suppose that there exists an ample-polarized (metrized) normal projective model $(X,\overline{L})=(X,L,h_{\mathrm{cscK}})$, whose 
$\omega_{X/\mathcal{O}_F}$ is $\mathbb{Q}$-Cartier, satisfying the following: 
\begin{enumerate}
\item (at complex place) 
\label{complex.asm}
$(X_{\mathbb{C}},L_{\mathbb{C}})$ has a 
constant scalar curvature K\"ahler metric 
$\omega_{h_{\mathrm{cscK}}}$ with the K\"ahler class 
$2\pi c_1(L_{\mathbb{C}})$. 

\item (on reductions) \label{poscharJ}
For each prime ideal $\mathfrak{p}$ of $\mathcal{O}_F$, 
$\overline{X_{\mathfrak{p}}}:=X_{\mathfrak{p}}\times_{\mathcal{O}_F/\mathfrak{p}}\overline{\mathcal{O}_F/\mathfrak{p}}$ is 
locally F-pure,  
$(X_{\mathfrak{p}},L_{\mathfrak{p}})$ is $\mathrm{J}^{K_{X_{\mathfrak{p}}}+\delta^F(\overline{L_{\mathfrak{p}}})L_{\mathfrak{p}}}$-semistable (cf.~Remark \ref{rem--j-for-positive}), 
and $K_{X_{\mathfrak{p}}}+\delta^F(\overline{L_{\mathfrak{p}}})L_{\mathfrak{p}}$ is 
nef. 
\end{enumerate}

Then, 
$h_K(X,L,h_{\mathrm{cscK}})$ attains the minimum among 
$h_K(X',L',h'_L)$ for all metrized ample polarized models $(X',L',h'_L)$ 
with the same generic fibers $(X_{\eta},L_{\eta})$. 
\end{Thm}



\begin{Rem}For the complex place, we remark that if $(X_{\mathbb{C}},L_{\mathbb{C}})$ is specially K-stable, then it implies the condition (i) i.e., there exists a unique metric $h_L$ such that $\omega_{h_L}$ has a constant scalar curvature by \cite[Corollary 5.2]{Zhang2} and \cite[Theorem 4.1]{CC}. 

  The assumption \eqref{poscharJ} is clearly a positive characteristic analogue of the special K-semistability (cf., Definition \ref{specialKss}). Therefore, roughly speaking, the above three assumptions are analogues of special K-(semi)stabilities for each place.  \end{Rem}

By \cite[\S 2]{OFal}, the obtained minimum gives 
a generalization of the Faltings height for abelian varieties (\cite{Fal}). 
Although the ``special K-stability" type 
assumptions in Theorem \ref{hK.min} on $(X,L)$ may look quite technical, 
many examples (compared with \cite[3.14]{OFal}) 
should satisfy as \cite[Theorem 3.12]{H22} summarizes 
(also cf., \S \ref{spKst.sec}, \cite{H21}). 

\begin{Ex}\label{ex}
\begin{enumerate}
\item 
Either if $X_\eta$ is smooth proper curve of genus 
$g\ge 2$ and all reductions are stable curves, 
or if $X_\eta$ is smooth elliptic curve and all reductions are $I_m$-type reductions for 
$m\ge 1$, then these classical examples of curves 
satisfy \eqref{hK.min}. 

\item 
In the case when $\mathrm{dim}\,X_{\eta}=2$, if $X_\eta$ (resp., $X_{\mathfrak{p}})$ 
is a smooth (resp., F-pure) minimal model, 
whose $L_\eta$ (resp., $L_{\mathfrak{p}}$)  is close enough to 
$K_{X_\eta}$ 
(resp., $K_{X_\mathfrak{p}}$) , 
the assumption 
\eqref{poscharJ} is satisfied 
by \cite[\S 8]{H21} and 
the 
assumption \eqref{complex.asm} also holds 
by {\it loc.cit}, \cite{Zhang2} 
(cf., also the earlier references therein. In \cite{H21}, we do not deal with the positive characteristic case, but we can also show the special K-stability of klt minimal models of dimension two).

\item \label{ex3}
For any projective module $\mathcal{E}$ 
over $\mathcal{O}_F$, 
$(\mathbb{P}(\mathcal{E}),\mathcal{O}(1))$ 
satisfies the above conditions of Theorem 
\ref{hK.min}
and hence satisfy the arithmetic 
Yau-Tian-Donaldson conjecture 
\ref{aYTD}, 
which is not confirmed in \cite{OFal}. 

\item \label{ex4}
More generally, we expect that most of 
K-stable Fano varieties over 
$\overline{\mathbb{Q}}$ have some polarized 
models which satisfies the conditions 
of Theorem 
\ref{hK.min}. 
 For instance, recently 
 the first author, S.~Pande, T.~Takamatsu confirmed that general del Pezzo surfaces $S$ of degree $1$, whose $|-K_S|$ contains only elliptic curves or rational curves with only one nodal singularity, have $\delta^F(S,-K_S)>1$. 
\end{enumerate}
\end{Ex}

To show Theorem \ref{hK.min}, we prepare a lemma below, which is a mixed characteristic analogue of \cite[Theorem 6.6]{BL} and \cite[Theorem 7.27]{Xu}.
\begin{Lem}
    Let $(X,\overline{L})$ be an ample-polarized (metrized) normal projective model.
Let $\mathfrak{p}$ be a prime ideal of $\mathcal{O}_F$ such that $\overline{X_{\mathfrak{p}}}$ is locally F-pure.
Then $\delta^F(\overline{L_{\mathfrak{p}}})\le \delta(X_{\mathbb{C}},L_{\mathbb{C}})$.\label{lem-blum-liu-toy}
\end{Lem}

\begin{proof}
This lemma follows from \cite[Corollary 3.9]{SatoTakagi} and a similar argument of \cite[Theorem 7.27]{Xu}, but we give a proof for the reader's convenience here.

For any $\epsilon>0$ and sufficiently divisible integer $k\in\mathbb{Z}_{>0}$, we have $\delta_k(X_{\mathbb{C}},L_{\mathbb{C}})\le\delta(X_{\mathbb{C}},L_{\mathbb{C}})+\epsilon$.
Let $\bar{F}$ be the algebraic closure of $F$ and $X_{\bar{F}}:=X_\eta\times_{\mathrm{Spec}\,F}{\mathrm{Spec}\,\bar{F}}$.
We note that $\delta_k(X_{\bar{F}},L_{\bar{F}})=\delta_k(X_{\mathbb{C}},L_{\mathbb{C}})$.
It is well-known to experts but we write the complete proof of this fact. First, it is easy to see that $\delta_k(X_{\bar{F}},L_{\bar{F}})\ge\delta_k(X_{\mathbb{C}},L_{\mathbb{C}})$ since the log canonicity is stable under changes of algebraically closed base fields. 
Next, we argue the converse inequality.
Recall that $\mathbb{P}^{h^0(X_{\bar{F}},L_{\bar{F}}^{\otimes k})-1}$ parametrizes effective divisors linearly equivalent to $L_{\bar{F}}^{\otimes k}$ and let $\mathcal{D}\subset X_{\bar{F}}\times \mathbb{P}^{h^0(X_{\bar{F}},L_{\bar{F}}^{\otimes k})-1}$ be the universal divisor.
Then, we set $\mathcal{D}'\subset X_{\bar{F}}\times (\mathbb{P}^{h^0(X_{\bar{F}},L_{\bar{F}}^{\otimes k})-1})^{\times h^0(X_{\bar{F}},L_{\bar{F}}^{\otimes k})}$ as
\[
\mathcal{D}':=\frac{1}{kh^0(X_{\bar{F}},L_{\bar{F}}^{\otimes k})}\sum_{i=1}^{h^0(X_{\bar{F}},L_{\bar{F}}^{\otimes k})}\mathrm{pr}_i^*\mathcal{D},
\]
where $\mathrm{pr}_i\colon X_{\bar{F}}\times (\mathbb{P}^{h^0(X_{\bar{F}},L_{\bar{F}}^{\otimes k})-1})^{\times h^0(X_{\bar{F}},L_{\bar{F}}^{\otimes k})}\to X_{\bar{F}}\times \mathbb{P}^{h^0(X_{\bar{F}},L_{\bar{F}}^{\otimes k})-1}$ is induced by the $i$-th projection of $(\mathbb{P}^{h^0(X_{\bar{F}},L_{\bar{F}}^{\otimes k})-1})^{\times h^0(X_{\bar{F}},L_{\bar{F}}^{\otimes k})}$, and 
\[
U:=\{s\in (\mathbb{P}^{h^0(X_{\bar{F}},L_{\bar{F}}^{\otimes k})-1})^{\times h^0(X_{\bar{F}},L_{\bar{F}}^{\otimes k})}|\,\mathcal{D}'_s\textrm{ is $k$-basis type}\}.
\]
It is easy to see that the fiber of $(X_{\bar{F}}\times (\mathbb{P}^{h^0(X_{\bar{F}},L_{\bar{F}}^{\otimes k})-1})^{\times h^0(X_{\bar{F}},L_{\bar{F}}^{\otimes k})},\delta_k(X_{\bar{F}},L_{\bar{F}})\mathcal{D}')$ over any geometric point $\bar{s}\in U$ is log canonical.
For any $k$-basis type divisor $D_{\mathbb{C}}$ in $L_{\mathbb{C}}$, we see that $(X_{\mathbb{C}},D_{\mathbb{C}})$ is the base change of some geometric fiber of $(X_{\bar{F}}\times (\mathbb{P}^{h^0(X_{\bar{F}},L_{\bar{F}}^{\otimes k})-1})^{\times h^0(X_{\bar{F}},L_{\bar{F}}^{\otimes k})},\delta_k(X_{\bar{F}},L_{\bar{F}})\mathcal{D}')$ over $(\mathbb{P}^{h^0(X_{\bar{F}},L_{\bar{F}}^{\otimes k})-1})^{\times h^0(X_{\bar{F}},L_{\bar{F}}^{\otimes k})}$ and hence log canonical.
Thus, we conclude that $\delta_k(X_{\bar{F}},L_{\bar{F}})\le\delta_k(X_{\mathbb{C}},L_{\mathbb{C}})$.

Let $E_{\bar{F}}$ be a prime divisor over $X_{\bar{F}}$ such that $E_{\bar{F}}$ computes $\delta_k(X_{\bar{F}},L_{\bar{F}})$, i.e., there exists a $k$-basis type divisor $D_{k,\bar{F}}\sim_{\mathbb{Q}}L_{\bar{F}}$ such that $(X_{\bar{F}},\delta_k(X_{\bar{F}},L_{\bar{F}})D_{k,\bar{F}})$ is log canonical and $E_{\bar{F}}$ is an lc place.
Then, we can take a finite field extension $K$ of $F$ such that $$\delta_k(X_{\mathbb{C}},L_{\mathbb{C}})=\delta_k(X_{\bar{F}},L_{\bar{F}})=\delta_k(X_K,L_K)$$ and there exists a prime divisor $E_K$ over $X_K:=X_\eta\times_{\mathrm{Spec}\,F}{\mathrm{Spec}\,K}$ that computes $\delta_k(X_K,L_K)$, where $L_K$ is the pullback of $L_\eta$.
Let $\mathcal{O}_K$ be the integral closure of $\mathcal{O}_F$ in $K$ and $\eta'$ the generic point of $\mathrm{Spec}\,(\mathcal{O}_K)$.
Let $\mathscr{F}_{E_K}$ be the filtration  of $H^0(X_K,L_K^{\otimes k})$ which is defined by $E_{K}$, that is
\[
\mathscr{F}_{E_K}^\lambda H^0(X_K,L_K^{\otimes k}):=H^0(X_K,L_K^{\otimes k}(-\lambda E_K))
\]
for any $\lambda\in\mathbb{Z}$.
Let $\pi_K\colon X\times_{\mathrm{Spec}\,(\mathcal{O}_F)}\mathrm{Spec}(\mathcal{O}_K)\to\mathrm{Spec}(\mathcal{O}_K)$ and $\mu\colon X\times_{\mathrm{Spec}\,(\mathcal{O}_F)}\mathrm{Spec}(\mathcal{O}_K) \to X$ be the canonical morphisms.
By the properness of the flag varieties and the fact that $\mathcal{O}_K$ is a Dedekind domain, we see that there exists a filtration $\mathscr{F}$ of the sheaf $\pi_*(\mathcal{O}_{X\times_{\mathrm{Spec}\,(\mathcal{O}_F)}\mathrm{Spec}(\mathcal{O}_K)}(k\mu^*L))$ such that $\mathscr{F}^\lambda\pi_*(\mathcal{O}_{X\times_{\mathrm{Spec}\,(\mathcal{O}_F)}\mathrm{Spec}(\mathcal{O}_K)}(k\mu^*L))|_{\eta'}=\mathscr{F}_{E_K}^\lambda H^0(X_K,L_K^{\otimes k})$ and 
\[
\mathscr{F}^\lambda\pi_*(\mathcal{O}_{X\times_{\mathrm{Spec}\,(\mathcal{O}_F)}\mathrm{Spec}(\mathcal{O}_K)}(k\mu^*L))/\mathscr{F}^{\lambda+1}\pi_*(\mathcal{O}_{X\times_{\mathrm{Spec}\,(\mathcal{O}_F)}\mathrm{Spec}(\mathcal{O}_K)}(k\mu^*L))
\]
is flat over $\mathrm{Spec}(\mathcal{O}_K)$ for any $\lambda\in\mathbb{Z}$.
Take a prime ideal $\mathfrak{p}'$ of $\mathcal{O}_K$ that is mapped to $\mathfrak{p}\in\mathrm{Spec}\,(\mathcal{O}_F)$. 
Then we can choose a free basis $\{s_1,\ldots,s_{h^0(X_K,L_K^{\otimes k})}\}$ of $\pi_*(\mathcal{O}_{X\times_{\mathrm{Spec}\,(\mathcal{O}_F)}\mathrm{Spec}(\mathcal{O}_K)}(k\mu^*L))\otimes_{\mathcal{O}_K}\mathcal{O}_{K,\mathfrak{p}'}$ such that for any $\lambda$, we can choose a subset of $\{s_1,\ldots,s_{h^0(X_K,L_K^{\otimes k})}\}$ that is a free basis of $\mathscr{F}^\lambda\pi_*(\mathcal{O}_{X\times_{\mathrm{Spec}\,(\mathcal{O}_F)}\mathrm{Spec}(\mathcal{O}_K)}(k\mu^*L))\otimes_{\mathcal{O}_K}\mathcal{O}_{K,\mathfrak{p}'}$.
Let $D=\frac{1}{kh^0(X_K,L_K^{\otimes k})}\sum_{j=1}^{h^0(X_K,L_K^{\otimes k})}\mathrm{div}(s_j)$ on $X\times_{\mathrm{Spec}\,(\mathcal{O}_F)}\mathrm{Spec}(\mathcal{O}_K)$. 
By taking $k$ large enough, we may assume that $H^1(X_\mathfrak{p},L^{\otimes k}_{\mathfrak{p}})=0$.
Then, $D_{\mathfrak{p'}}$ and $D_{K}$ are $k$-basis type divisors
. By the choice of $D$ and the proof of \cite[Lemma 2.2]{FO}, we see that $D_K$ attains $\max_{D'_K}\mathrm{ord}_{E_K}(D'_K)$, where $D'_K$ runs over all $k$-basis type divisors, and hence $$\delta_k(X_K,L_K)=\mathrm{lct}(X_K;D_K)=\frac{A_{X_K}(E_K)}{\mathrm{ord}_{E_K}(D_K)}.$$
Therefore, it follows from \cite[Corollary 3.9]{SatoTakagi} that $\mathrm{Fpt}(X_{\mathfrak{p'}},D_{\mathfrak{p'}})\le \delta_k(X_K,L_K)$, which means that 
\[
\delta^F_{X_{\mathfrak{p}'},k}(L_{\mathfrak{p}'})\le\delta_k(X_K,L_K)= \delta_k(X_{\mathbb{C}},L_{\mathbb{C}}).
\]
By the definition of F-purity (cf.~\cite[Definition 2.7]{SatoTakagi}) and Definition \ref{Frob.delta}, we have
\[
\delta^F_{k}(\overline{L_{\mathfrak{p}}})\le\delta^F_{X_{\mathfrak{p}'},k}(L_{\mathfrak{p}'}).
\]
This shows $$\delta^F_{k}(\overline{L_{\mathfrak{p}}})\le\delta_k(X_{\mathbb{C}},L_{\mathbb{C}})\le \delta(X_{\mathbb{C}},L_{\mathbb{C}})+\epsilon$$ for any $\epsilon>0$ and sufficiently large $k$.
 Therefore, we have 
 \[
 \delta^F(\overline{L_{\mathfrak{p}}})=\liminf_{k\to\infty}\delta^F_{k}(\overline{L_{\mathfrak{p}}})\le\delta(X_{\mathbb{C}},L_{\mathbb{C}})+\epsilon.
 \]
Thus, we have $ \delta^F(\overline{L_{\mathfrak{p}}})\le\delta(X_{\mathbb{C}},L_{\mathbb{C}})$ and complete the proof.
\end{proof}

\begin{Rem}\label{rem:deltacompare}
    By the same argument as the above proof, we see that $ \delta^F(L_{\mathfrak{p}})\le\delta(X_{\eta},L_{\eta})$.
    However, we cannot replace $\delta^F(\overline{L_{\mathfrak{p}}})$ with $ \delta^F(L_{\mathfrak{p}})$ in the statement of Lemma \ref{lem-blum-liu-toy} since the inequality $\delta(X_{\eta},L_{\eta})\ge\delta(X_{\mathbb{C}},L_{\mathbb{C}})$ could be strict in general.
Indeed, we have the following example. Let $K$  be a finitely generated field over $\mathbb{C}$ whose transcendence degree is two.
It is well-known (cf.~\cite[III, Exercise 9.10 (b)]{Harts}) that there exists a proper smooth variety $X_K$ over $K$ with ample $-K_{X_K}$ such that $X_K\times_{\mathrm{Spec}(K)}\mathrm{Spec}(\overline{K})\cong \mathbb{P}^1_{\overline{K}}$ but $X_K\not\cong\mathbb{P}^1_K$.
Even though the K-semistability of $(X_K,-K_{X_K})$ and $(\mathbb{P}^1_{\bar{K}},-K_{\mathbb{P}^1_{\bar{K}}})$ are equivalent by \cite[Theorem 1.1]{Zhu}, we have $\delta(X_K,-K_{X_K})\ne\delta(\mathbb{P}^1_{\bar{K}},-K_{\mathbb{P}^1_{\bar{K}}})$ in this case.
We see this fact as follows.
It is not hard to see that $E\times_{\mathrm{Spec}(K)}\mathrm{Spec}(\overline{K})$ is a union of distinct $n_E$-points on $\mathbb{P}^1_{\bar{K}}$ for every prime divisor $E$ over $X_K$.
Since $X_K\not\cong\mathbb{P}^1_K$, we have that $n_E>1$.
This means that for any nonzero section $s\in H^0(X_K,-mK_{X_K})$, the pullback of $s$ to $\mathbb{P}^1_{\bar{K}}$ has the same vanishing order on each point of $E\times_{\mathrm{Spec}(K)}\mathrm{Spec}(\overline{K})$.
Therefore,
\[
\delta(X_K,-K_{X_K})=\inf_{E}\frac{\mathrm{deg}(-K_{X_K})\cdot A_{X_K}(E)}{\int^{\infty}_{0}\mathrm{vol}\,(-K_{X_K}-xE)dx}= \inf_E\frac{2}{\int^{\frac{2}{n_E}}_{0}(2-n_Ex)dx}=\inf_En_E>1,
\]
where $E$ runs over all prime divisor over $X_K$ and note that $\mathrm{vol}\,(-K_{X_K}-xE)=\max\{0,\mathrm{deg}\,(-K_{X_K}-xE)\}$.
Here, we used the fact that $A_{X_K}(E)=1$ and \cite[Corollary 3.9, Theorem 4.4]{BJ}.
On the other hand, it is well-known that $\delta(\mathbb{P}^1_{\bar{K}},-K_{\mathbb{P}^1_{\bar{K}}})=1$.
Thus, $\delta(\mathbb{P}^1_{\bar{K}},-K_{\mathbb{P}^1_{\bar{K}}})<\delta(X_K,-K_{X_K})$.

\end{Rem}

We prepare the following application of the Fujita vanishing theorem (cf.~\cite[Theorem 3.8.1]{Fuj}).
\begin{Lem}\label{lem-fujita}
    Let $\mathcal{F}$ be a coherent sheaf on $\mathbb{P}^n_A$, where $A$ is an Artinian local ring.
    Let $H$ be an ample line bundle on $\mathbb{P}^n_A$.
    Then we obtain the following.
    \begin{enumerate}
        \item There exists $m\in\mathbb{Z}_{>0}$ depending only on $\mathcal{F}$ such that $H^j(\mathbb{P}^n_A,\mathcal{F}(mH+D))=0$ for any $j>0$ and nef Cartier divisor $D$, and
        \item for any nef Cartier divisor $D$, we have
        \[
        \mathrm{length}_A(H^j(\mathbb{P}^n_A,\mathcal{F}(mD)))=O(m^{l-j})
        \]
        for any $j$ and sufficiently large $m$, where $l=\mathrm{dim}\,\mathrm{Supp}(\mathcal{F})$. Here, $\mathrm{length}_A$ denotes the length of an $A$-module.
    \end{enumerate}
\end{Lem}

\begin{proof}
    First, we note that if $A$ is a field, then both (i) and (ii) hold.
    Indeed, (i) and (ii) are shown (cf.~\cite[3.8.1, 3.9.1]{Fuj}) when $A$ is an algebraically closed field.
    If $A$ is not algebraically closed, we conclude that (i) and (ii) also hold in this case by changing the base field $A$ to an algebraically closed field.

For general case, consider the following short exact sequence
\[
0\to\mathfrak{m}^n\mathscr{F}\to\mathfrak{m}^{n-1}\mathscr{F}\to \mathfrak{m}^{n-1}\mathscr{F}/\mathfrak{m}^n\mathscr{F}\to0
\]
for $n\in\mathbb{Z}_{>0}$, where $\mathfrak{m}$ is the maximal ideal of $A$.
Since $A$ is Artinian, $\mathfrak{m}^n\mathscr{F}=0$ for some $n$ and then $\mathfrak{m}^{n-1}\mathscr{F}$ is a sheaf on $\mathbb{P}^n_{A/\mathfrak{m}}$.
It is easy to see the following for any $n$:
\begin{itemize}
    \item if (i) and (ii) hold for $\mathfrak{m}^n\mathscr{F}$ and $\mathfrak{m}^{n-1}\mathscr{F}/\mathfrak{m}^n\mathscr{F}$, then (i) and (ii) also hold for $\mathfrak{m}^{n-1}\mathscr{F}$.
\end{itemize}
Therefore, (i) and (ii) hold for $\mathscr{F}$ by the induction on $n$.
\end{proof}

\begin{proof}[Proof of Theorem \ref{hK.min}]
Take any other positively metrized ample polarized model 
$(X',L',h')$ whose generic fiber is the same i.e., $(X_{\eta},L_{\eta})$. 
Then we can take a finite sequence of metrized polarized models 
$(X(k),L(k),h(k))$ for $k=0,\cdots,m$ such that 
\begin{enumerate}
\item $(X(0),L(0),h(0))=(X,L,h_{\mathrm{cscK}})$, 
\item $(X(1),L(1),h(1))=(X,L,h')$, 
\item $(X(m),L(m),h(m))=(X',L',h')$,
\item For each $k\ge 1$, $(X(k),L(k),h(k))$ and $(X(k+1),L(k+1),h(k+1))$ differs exactly at one non-archimedean 
place of $F$, which corresponds to $\mathfrak{p}_k\subset 
\mathcal{O}_F$ and are all distinct. 
\end{enumerate}
Note that 
$L(i)_{\eta}$ is independent of $i$. 

Firstly, we have 
\begin{align}
h_K(X(0),L(0),h(0)=h_{\mathrm{cscK}})\le 
h_K(X(1),L(1),h(1)=h'),\label{key-anal}
\end{align}
because of the change of metric formula 
(cf., e.g., \cite[2.2]{OFal}) and 
the assumption that $\omega_{h_{\mathrm{cscK}}}$ is a cscK metric 
which minimizes the (complex) Mabuchi's 
K-energy. 

Next, we deal with the inequality
\begin{equation}
h_K(X(k+1),L(k+1),h(k+1))\ge h_K(X(k),L(k),h(k))\label{eq--keyk}
\end{equation}
for any $k\ge1$, which completes the proof of Theorem \ref{hK.min}.
Indeed, combining \eqref{eq--keyk} with \eqref{key-anal}, we have that 
\begin{align*}
    &h_K(X(m),L(m),h(m))-h_K(X(0),L(0),h(0))\\
    &=\sum^{m-1}_{k=0} h_K(X(k+1),L(k+1),h(k+1))-h_K(X(k),L(k),h(k))\ge0.
\end{align*}
First, we deal with \eqref{eq--keyk} in the case when $\delta^F(\overline{L_{\mathfrak{p}_k}})>0$.
  

{\it Case~1.~$\delta^F(\overline{L_{\mathfrak{p}_k}})>0$}.~In this section, to show \eqref{eq--keyk} for any $k\ge1$, we change the models and reduce the argument to comparing Arakelov-Ding functionals and J-energies in the case when $\delta^F(\overline{L_{\mathfrak{p}_k}})>0$.
Here, we note that $h(k+1)=h(k)$. 
Take a normalized blow up $\nu\colon X'(k+1)\to X(k+1)$ along some closed subschemes supported on $X(k+1)_{\mathfrak{p}}$ such that there exists a proper birational morphism $\mu\colon X'(k+1)\to X(k)$.
We construct a model $(X''(k+1),L''(k+1),h(k))$ (resp.~$(X''(k),L''(k),h(k))$), where there exists a canonical projective birational morphism $\mu''\colon X''(k+1)\to X''(k)$, by patching $(X(0),L(0),h(k))\times_{\mathrm{Spec}\,\mathcal{O}_F}(\mathrm{Spec}\,\mathcal{O}_F\setminus\{\mathfrak{p}_{k}\})$ and $(X'(k+1),\nu^*L(k+1),h(k))\times_{\mathrm{Spec}\,\mathcal{O}_F}(\mathrm{Spec}\,\mathcal{O}_F\setminus\{\mathfrak{p}_{j}\}_{j\ne k})$ (resp.~$(X(k),L(k),h(k))\times_{\mathrm{Spec}\,\mathcal{O}_F}(\mathrm{Spec}\,\mathcal{O}_F\setminus\{\mathfrak{p}_{j}\}_{j\ne k})$) together over $\mathrm{Spec}\,\mathcal{O}_F\setminus\{\mathfrak{p}_1,\ldots,\mathfrak{p}_{m}\}$.
Then, it is easy to see that
\begin{align}
& h_{K}(X(k+1),L(k+1),h(k))-h_{K}(X(k),L(k),h(k))\label{eq--reduction--1}\\
  = & h_{K}(X'(k+1),\nu^*L(k+1),h(k))-h_{K}(X(k),L(k),h(k))\nonumber\\
   =&h_{K}(X''(k+1),L''(k+1),h(k))-h_K(X''(k),L''(k),h(k))\nonumber.
\end{align}
On the other hand, for any sufficiently small $\gamma>0$ such that $\delta^F(\overline{L_{\mathfrak{p}_k}})-\gamma\in\mathbb{Q}_{>0}$, by replacing $\bar{L}$ with $(\delta^F(\overline{L_{\mathfrak{p}_k}})-\gamma)\bar{L}$ and setting $\epsilon:=\frac{2\gamma}{\delta^F(\overline{L_{\mathfrak{p}_k}})-\gamma}$, we may assume that $\delta^F(\overline{L_{\mathfrak{p}_k}})>1$ and $H(k):=K_{X''(k)}+(1+\epsilon)L''(k)$ is ample on $X''_{\mathfrak{p}_k}$
.
Then we also have $\delta(X_{\mathbb{C}},L_{\mathbb{C}})>1$
by Lemma \ref{lem-blum-liu-toy}. 
Take an arbitrary hermitian metric $h_{H(k)}$ on $H(k)_{\mathbb{C}}$ and set $\overline{H(k)}=(H(k),h_{H(k)})$.
By \cite[Theorem 2.3]{Zhang2}, we have a unique $-\omega_{h_{H(k)}}+\epsilon \omega_{h(k)}$-twisted K\"{a}hler-Einstein metric $\omega_{h_{L}(k)}$ in $2\pi c_1(L_{\mathbb{C}})$. 
Now, we claim the following. 
\begin{claim}
    Suppose that
    \begin{equation}
h_{K,\epsilon (\mu''^*L''(k),h_L(k))}(X''(k+1),L''(k+1),h_L(k))-h_{K,\epsilon(L''(k),h_L(k))}(X''(k),L''(k),h_L(k))
\label{eq--keytwist}
\end{equation}
is nonnegative for any sufficiently small $\gamma$ and $\epsilon>0$.
Then, the inequality (\ref{eq--keyk}) holds.\label{claim-1}
\end{claim}
\begin{proof}
Assume that Claim \ref{claim-1} fails.
Then 
\[
h_{K}(X(k+1),L(k+1),h(k))-h_{K}(X(k),L(k),h(k))<0.
\]
Note that 
\begin{align*}
    &\lim_{\epsilon\to0}\bigl(h_{K,\epsilon (\mu''^*L''(k),h(k))}(X''(k+1),L''(k+1),h(k))\\
    &-h_{K,\epsilon(L''(k),h(k))}(X''(k),L''(k),h(k))\bigr)\\
    =&h_{K}(X''(k+1),L''(k+1),h(k))-h_{K}(X''(k),L''(k),h(k)).
\end{align*}
By the above equation and \eqref{eq--reduction--1}, we can take sufficiently small $\epsilon$ such that 
\begin{equation}
h_{K,\epsilon (\mu''^*L''(k),h(k))}(X''(k+1),L''(k+1),h(k))-h_{K,\epsilon(L''(k),h(k))}(X''(k),L''(k),h(k))<0.\label{eq--reduction--2}
\end{equation}
By the change of metric formula (cf., e.g., \cite[2.2]{OFal}), we have that
\begin{align}
    &h_{K,\epsilon (\mu''^*L''(k),h_{L}(k))}(X''(k+1),L''(k+1),h_{L}(k))-h_{K,\epsilon(L''(k),h_{L}(k))}(X''(k),L''(k),h_{L}(k))\nonumber
    \\
    =&h_{K,\epsilon (\mu''^*L''(k),h(k))}(X''(k+1),L''(k+1),h(k))-h_{K,\epsilon(L''(k),h(k))}(X''(k),L''(k),h(k))\nonumber
\end{align}
for any $\epsilon\ge0$.
By the above equation and \eqref{eq--reduction--2}, we have that
$$h_{K,\epsilon (\mu''^*L''(k),h_L(k))}(X''(k+1),L''(k+1),h_L(k))-h_{K,\epsilon(L''(k),h_L(k))}(X''(k),L''(k),h_L(k))<0.$$
This contradicts to the assumption that
\eqref{eq--keytwist} is nonnegative
for any sufficiently small $\epsilon>0$.
We complete the proof of Claim \ref{claim-1}.
\end{proof}
From now, we fix a sufficiently small $\epsilon>0$ and deal with \eqref{eq--keytwist}.
By Lemma \ref{2.12.analogue}, $K_{X''(k)}+\epsilon L''(k)-H(k)=- L''(k)$ and the property of $\omega_{h_L(k)}$, we have that 
\begin{align*}
h_{K,\epsilon (\mu''^*L''(k),h_L(k))}(X''(k+1)&,L''(k+1),h_L(k))\\
&\ge\mathcal{J}^{Ar,\mu''^*\overline{H(k)}}(X''(k+1),L''(k+1),h_L(k))\\
&+\mathcal{D}_{\mu''^*(\epsilon \overline{L''(k)}-\overline{H(k)})}^{Ar}(X''(k+1),L''(k+1),h_L(k))\\
&-(L_{\eta}^n)\log (L_{\eta}^n),\quad \textrm{and}\\
		   h_{K,\epsilon(L''(k),h_L(k))}(X''(k),L''(k),h_L(k)) &=\mathcal{J}^{Ar,\overline{H(k)}}(X''(k),L''(k),h_L(k))\\
     &+\mathcal{D}_{\epsilon \overline{L''(k)}-\overline{H(k)}}^{Ar}(X''(k),L''(k),h_L(k))\\
     &-(L_{\eta}^n)\log (L_{\eta}^n).
\end{align*}
Therefore,
\begin{align*}
&(\ref{eq--keytwist})
\ge\mathcal{J}^{Ar,\mu''^*\overline{H(k)}}(X''(k+1),L''(k+1),h_L(k))-\mathcal{J}^{Ar,\overline{H(k)}}(X''(k),L''(k),h_L(k))\\
		    &+\mathcal{D}_{\mu''^*(\epsilon \overline{L''(k)}-\overline{H(k)})}^{Ar}(X''(k+1),L''(k+1),h_L(k))-\mathcal{D}_{\epsilon \overline{L''(k)}-\overline{H(k)}}^{Ar}(X''(k),L''(k),h_L(k)).
\end{align*}
To show \eqref{eq--keytwist} is nonnegative, it suffices to show the following values are nonnegative:
\begin{align}
    \mathcal{J}^{Ar,\mu''^*\overline{H(k)}}(X''(k+1),L''(k+1),h_L(k))&-\mathcal{J}^{Ar,\overline{H(k)}}(X''(k),L''(k),h_L(k))\label{2}\\
    \mathcal{D}_{\mu''^*(\epsilon \overline{L''(k)}-\overline{H(k)})}^{Ar}(X''(k+1),L''(k+1),h_L(k))&-\mathcal{D}_{\epsilon \overline{L''(k)}-\overline{H(k)}}^{Ar}(X''(k),L''(k),h_L(k)).\label{4}
\end{align}

Next, we apply 
the same arguments as \cite[3.15]{H22} (comparing twisted Arakelov J-energy) 
to show
\eqref{2}
 is nonnegative as we recap as follows. 
 First, let $f(k)\colon X''(k)\to\mathrm{Spec}(\mathcal{O}_F)$ be the canonical morphism. 
 As \cite[3.15]{H22}, we may assume that $E:=\mu''^*L(k)-L''(k+1)$ is an effective divisor supported on $X''(k+1)_{\mathfrak{p}_k}$.
 We note that $(X_{\mathfrak{p}_k},L_{\mathfrak{p}_k})$ is $\mathrm{J}^{H_{\mathfrak{p}_k}}$-semistable and $H_{\mathfrak{p}_k}$ is ample by the choice of $H_{\mathfrak{p}_k}$.
After \cite{H22}, we construct a filtration $\mathscr{F}$ for $(X''(k)_{\mathfrak{p}},L''(k)_{\mathfrak{p}})=(X(k)_{\mathfrak{p}},L(k)_{\mathfrak{p}})$ 
from 
$(X''(k), L''(k))$ and $(X''(k+1),L''(k+1))$: for each $m\ge 0$, we first take the filtration of $f(k)_*(L''(k)^{\otimes m})$
\begin{equation}\label{eq}
  \mathcal{F}^{i}f(k)_*(L''(k)^{\otimes m}):=\left\{
\begin{split}
\mathfrak{p}_k^{i}((f(k)\circ\mu'')_*L''(k+1)^{\otimes m})\cap f(k)_*(L''(k)^{\otimes m}) \quad \mathrm{for} \, i\le 0 \\
0\quad \mathrm{for} \, i> 0,
\end{split}\right.
  \end{equation}
and then we set $$\mathscr{F}^{i}H^0(X(k)_{\mathfrak{p}_{k}},L(k)_{\mathfrak{p}_{k}}^{\otimes m})$$ as the images of $\mathcal{F}^{i}f(k)_*(L''(k)^{\otimes m})\to H^0(X(k)_{\mathfrak{p}_k},L(k)_{\mathfrak{p}_k}^{\otimes m})$. 
It is easy to check that $\mathscr{F}$ satisfies Definition \ref{def--j-filt}.
We set as in \cite[Theorem 3.5]{H22} the following value 
$$w_{\mathscr{F}}(m):=\sum_{i=-\infty}^{\infty}i\cdot\mathrm{length}_{\mathcal{O}_F}(\mathscr{F}^{i}H^0(X(k)_{\mathfrak{p}_{k}},L(k)_{\mathfrak{p}_{k}}^{\otimes m})/\mathscr{F}^{i+1}H^0(X(k)_{\mathfrak{p}_{k}},L(k)_{\mathfrak{p}_{k}}^{\otimes m})).$$
We note that all but finitely many terms in the above sum are zero.
On the other hand, the value \eqref{2} equals to
\begin{align*}
\frac{1}{[F:\mathbb{Q}]L_{\eta}^{n}}\Biggl(&-E\cdot\left(\sum_{j=0}^{n-1}L''(k+1)^j\cdot \mu''^*L''(k)^{n-1-j}\right)\\
&+\frac{nH_{\eta}\cdot L_{\eta}^{n-1}}{(n+1)L_{\eta}^{n}}E\cdot\left(\sum_{j=0}^{n}L''(k+1)^j\cdot \mu''^*L''(k)^{n-j}\right)\Biggr).
\end{align*}
Since the support of $E$ is proper, the above intersection numbers are well-defined.
To show this value is nonnegative, we may assume that $L''(k+1)$ is relatively ample by perturbing the coefficients of $E$. Then, we note that the following claim holds as \cite[Theorem 3.5]{H22}.
We remark that we cannot directly apply \cite[Theorem 3.5]{H22} to obtain the following claim since we assumed there that the base curve $C$ is proper.
\begin{claim}
\label{claim-2}    $$\lim_{m\to\infty}\frac{(n+1)!}{m^{n+1}}w_{\mathscr{F}}(m)=-E\cdot\left(\sum_{j=0}^{n}L''(k+1)^j\cdot \mu''^*L''(k)^{n-j}\right).$$
\end{claim}
\begin{proof}
   As the proof of \cite[Theorem 3.5]{H22}, we see that $$w_{\mathscr{F}}(m)=-\mathrm{length}_{\mathcal{O}_F}(f(k)_*(L''(k)^{\otimes m})/(f(k)\circ\mu'' )_*(L''(k+1)^{\otimes m}))$$ for any sufficiently large and divisible $m\in\mathbb{Z}_{>0}$.
   Thus, it suffices to show that 
   \begin{align}
       &\lim_{m\to\infty}\frac{(n+1)!}{m^{n+1}}\mathrm{length}_{\mathcal{O}_F}(f(k)_*(L''(k)^{\otimes m})/(f(k)\circ\mu'')_*(L''(k+1)^{\otimes m}))\label{eq--dim-comparison}\\
       =&E\cdot\left(\sum_{j=0}^{n}L''(k+1)^j\cdot \mu''^*L''(k)^{n-j}\right).\nonumber
   \end{align}
   Note that there exists the following exact sequence of coherent sheaves on $X''(k+1)$ for any $i$ and $m$,
   \[
   0\to\mu''^*L''(k)^{\otimes m}(-(i+1)E)\to\mu''^*L''(k)^{\otimes m}(-iE)\to\mu''^*L''(k)^{\otimes m}(-iE)|_E\to0.
   \]
   Note also that the schematic image structure of $f(k)(E)$ is an Artinian scheme.
   By Lemma \ref{lem-fujita} applied to $(\mu''^*L''(k)^{\otimes m}(-iE))_{\mathfrak{p}_k}$ and $E$, there exists $N>0$ such that $R^j(f(k)\circ\mu'')_*(\mu''^*L''(k)^{\otimes m}(-iE))=0$ around $\mathfrak{p}_k$ and $H^j(E,\mu''^*L''(k)^{\otimes m}(-iE)|_E)=0$ for any sufficiently large $m$, $j>0$ and $N\le i\le m$. Thus, the following injective homomorphism
   \begin{align*}
   &(f(k)\circ\mu'')_*(\mu''^*L''(k)^{\otimes m}(-iE))/(f(k)\circ\mu'')_*(\mu''^*L''(k)^{\otimes m}(-(i+1)E))\\
   &\hookrightarrow H^0(E,\mu''^*L''(k)^{\otimes m}(-iE)|_E)
   \end{align*}
   is bijective and $$\mathrm{length}_{\mathcal{O}_F}(H^0(E,\mu''^*L''(k)^{\otimes m}(-iE)|_E))=\chi(E,\mu''^*L''(k)^{\otimes m}(-iE)|_E)$$ for any $N\le i\le m-1$ and sufficiently large $m$.
   Here, we set 
   \[
   \chi(E,\mu''^*L''(k)^{\otimes m}(-iE)|_E):=\sum_{j=0}^{n}(-1)^j\mathrm{length}_{\mathcal{O}_F}(H^j(E,\mu''^*L''(k)^{\otimes m}(-iE)|_E)).
   \]
   On the other hand, we apply Lemma \ref{lem-fujita} to $H^j(E,\mu''^*L''(k)^{\otimes m}(-iE)|_E)$ and obtain that
  \begin{align*}
  \mathrm{length}_{\mathcal{O}_F}(H^0(E,\mu''^*L''(k)^{\otimes m}(-iE)|_E))=\chi(E,\mu''^*L''(k)^{\otimes m}(-iE)|_E)+O(m^{n-1})
   \end{align*}
   for any $i$ and sufficiently large $m$.
   Note that $\chi(E,\mu''^*L''(k)^{\otimes m}(-iE)|_E)$ is a polynomial of $m$ and $i$ of degree $n$ with the leading term $\frac{m^n}{n!}\left(\mu''^*L''(k)-\frac{i}{m}E\right)^n\cdot E$ (cf.~\cite[Appendix B]{FGA}).
   It means that 
   \begin{align*}
       &\mathrm{length}_{\mathcal{O}_F}(f(k)_*(L''(k)^{\otimes m})/(f(k)\circ\mu'')_*(L''(k+1)^{\otimes m}))\\
       &=\sum_{i=0}^{m-1} \mathrm{length}_{\mathcal{O}_F}(f(k)\circ\mu'')_*(\mu''^*L''(k)^{\otimes m}(-iE))/(f(k)\circ\mu'')_*(\mu''^*L''(k)^{\otimes m}(-(i+1)E))\\
       &=\sum_{i=0}^{m-1}\frac{m^n}{n!}\left(\mu''^*L''(k)-\frac{i}{m}E\right)^n\cdot E+O(m^{n}).
   \end{align*}
   By the dominated convergence theorem, we obtain
   \begin{align*}
   &\lim_{m\to\infty}\frac{(n+1)!}{m^{n+1}}\left(\sum_{i=0}^{m-1}\frac{m^n}{n!}\left(\mu''^*L''(k)-\frac{i}{m}E\right)^n\cdot E+O(m^{n})\right)\\
   =&(n+1)\int^{1}_0\left(\mu''^*L''(k)-xE\right)^n\cdot Edx\\
   =&E\cdot\left(\sum_{j=0}^{n}L''(k+1)^j\cdot \mu''^*L''(k)^{n-j}\right),
   \end{align*}
  which shows \eqref{eq--dim-comparison}.
\end{proof}
Take a sufficiently large integer $a>0$ such that $aH(k)_{\mathfrak{p}_k}$ is very ample. 
Take a discrete valuation ring $R$ dominating $\mathcal{O}_{F,\mathfrak{p}_k}$ such that the residue field $R/\mathfrak{m}_R$ of $R$ is an uncountable algebraically closed field by \cite[Theorem 83]{Ma}.
Here, $\mathfrak{m}_R$ is the maximal ideal of $R$.
Let $X''(k)_R:=X''(k)\times_{\mathrm{Spec}(\mathcal{O}_{F})}\mathrm{Spec}(R)$, $X''(k+1)_R:=X''(k+1)\times_{\mathrm{Spec}(\mathcal{O}_{F})}\mathrm{Spec}(R)$ and $g_R\colon X''(k)_R\to X''(k)$ be the canonical morphism.
Then, we can take a general section $s\in H^0(X''(k)_R,\mathcal{O}_{X''(k)_R}(ag_R^*H(k)))$ such that $\mathrm{div}(s)$ satisfying the following (by the Bertini theorem for very ample divisors and the fact that $R/\mathfrak{m}_R$ is uncountable):
\begin{itemize}
    \item  $\mathrm{div}(s)|_{\mathfrak{m}_R}$ is reduced,
    \item $\mathrm{div}(s)|_{\mathfrak{m}_R}$ is compatible with an approximation $\{\mathscr{F}_{(l)}\}_{l>0}$ of the filtration $\mathscr{F}_R$, which is defined by $$\mathscr{F}_R^{\lambda}H^0((X''(k)_R)_{\mathfrak{m}_R},(L''(k)_R)_{\mathfrak{m}_R}^{\otimes m}):=\mathscr{F}^{\lambda}H^0(X(k)_{\mathfrak{p}_k},L(k)_{\mathfrak{p}_k}^{\otimes m})\otimes_{\mathcal{O}_F/\mathfrak{p}_k}(R/\mathfrak{m}_R)$$
    (for the definition of the approximation $\{\mathscr{F}_{(l)}\}_{l>0}$, we refer to Definition \ref{def--j-filt}), and
    \item the support of $\mu_R''^*\mathrm{div}(s)$ contains no $\mu_R''$-exceptional divisor, where $\mu_R''\colon X''(k+1)_R\to X''(k)_R$ is the morphism induced by $\mu''$.
\end{itemize}
The last condition implies that $\mu_R''^*\mathrm{div}(s)$ is reduced at all points of $(X''(k+1)_R)_{\mathfrak{m}_R}$ of codimension one.
Thus, the reduced structure $\mu_R''^*\mathrm{div}(s)_{\mathrm{red}}$ of $\mu_R''^*\mathrm{div}(s)$ is flat over $R$ and isomorphic to $\mathrm{div}(s)$ over the generic point of $\mathrm{Spec}(R)$.
We construct a filtration $\mathscr{F}_{\mathrm{div}(s)}H^0(\mathrm{div}(s)_{\mathfrak{m}_R},(L''(k)_R)^{\otimes m}|_{\mathrm{div}(s)_{\mathfrak{m}_R}})$ for $\mathrm{div}(s)$ and $\mu_R''^*\mathrm{div}(s)_{\mathrm{red}}$ as \eqref{eq}.
By Claim \ref{claim-2} applied to $\mathscr{F}_{\mathrm{div}(s)}$, we obtain that
\[
\eqref{2}=\lim_{m\to\infty}\frac{1}{[F:\mathbb{Q}]L_{\eta}^{n}}\left(\frac{n!w_{\mathscr{F}_{\mathrm{div}(s)}}(m)}{am^n}-\frac{nH_{\eta}\cdot L_{\eta}^{n-1}}{(n+1)L_{\eta}^{n}}\frac{(n+1)!w_{\mathscr{F}}(m)}{m^{n+1}}\right).
\]
Then, the same discussion as \cite[3.8, 3.15]{H22} shows that (for the definition of $\mathcal{J}^{H,\mathrm{NA}}(\mathscr{F}_R)$, we refer to Lemma \ref{lem-j-filt})
\[
\lim_{m\to\infty}\left(\frac{n!w_{\mathscr{F}_{\mathrm{div}(s)}}(m)}{am^n}-\frac{nH_{\eta}\cdot L_{\eta}^{n-1}}{(n+1)L_{\eta}^{n}}\frac{(n+1)!w_{\mathscr{F}}(m)}{m^{n+1}}\right)\ge\mathcal{J}^{H,\mathrm{NA}}(\mathscr{F}_R).
\]
By the construction of $\mathscr{F}_R$, we have $\mathcal{J}^{H,\mathrm{NA}}(\mathscr{F}_R)=\mathcal{J}^{H,\mathrm{NA}}(\mathscr{F})$.
On the other hand, $\mathcal{J}^{H,\mathrm{NA}}(\mathscr{F})\ge0$ by Lemma \ref{lem-j-filt}.
Summarizing them, we obtain
\[
\eqref{2}\ge\frac{1}{[F:\mathbb{Q}]L_{\eta}^{n}}\mathcal{J}^{H,\mathrm{NA}}(\mathscr{F}_R)\ge0.
\]

Finally, we apply the same arguments as \cite[3.19]{H22}
to show 
\eqref{4}
is nonnegative  
by using a recent variant of inversion of adjunction due to \cite[Theorem 3.8]{SatoTakagi} via 
the 
theory of F-singularities. 
More precisely speaking, we discuss as follows. 
By making use of Claim \ref{claim-2} instead of \cite[Theorem 3.5]{H22}, we apply the same argument as the proof of \cite[3.19]{H22} and obtain the following estimate:
\begin{align*}
    \eqref{4}\ge\liminf_{l\to\infty}\inf_{D}\mathrm{lct}(X,X_{\mathfrak{p}_k}+D;X_{\mathfrak{p}_k}),
\end{align*}
where $D$ runs over all effective $\mathbb{Q}$-Cartier $\mathbb{Q}$-divisors such that the support of $D$ does not contain $X_{\mathfrak{p}_k}$ and $D_{\mathfrak{p}_k}$ is an $l$-basis type divisor with respect to $L''(k)_{\mathfrak{p}_k}$ where $l$ is sufficiently large.
Since we assumed that $\delta^F(\overline{L_{\mathfrak{p}_k}})>1$ in the fourth paragraph of this proof, $\delta^F(L_{\mathfrak{p}_k})>1$ also holds.
 Therefore, it follows from 
\cite[Theorem 3.8]{SatoTakagi} that 
$(X''(k),X''(k)_{\mathfrak{p}_k}+D)$ is 
log canonical 
for any effective $\mathbb{Q}$-Cartier $\mathbb{Q}$-divisor $D$ whose support does not contain $X_{\mathfrak{p}_k}$ and whose restriction $D_{\mathfrak{p}_k}$ to $X''(k)_{\mathfrak{p}_k}$ is an $l$-basis type divisor with respect to $L''(k)_{\mathfrak{p}_k}$ where $l$ is sufficiently large.  
Thus, we have
\[
\liminf_{l\to\infty}\inf_{D}\mathrm{lct}(X,X_{\mathfrak{p}_k}+D;X_{\mathfrak{p}_k})\ge0,
\]
which shows that \eqref{4} is nonnegative.
Since \eqref{2} and \eqref{4} are nonnegative, so is \eqref{eq--keytwist}.
Therefore, we complete the proof that 
\eqref{eq--keyk} holds in Case~1.

{\it Case~2}.~We deal with the case when $\delta^F(\overline{L_{\mathfrak{p}_k}})=0$.
By \eqref{eq--reduction--1}, we may replace the models $(X(k),L(k),h(k))$ and $(X(k+1),L(k+1),h(k))$ with $(X''(k),L''(k),h(k))$ and $(X''(k+1),L''(k+1),h(k))$ respectively.
We have that
\begin{align*}
    &\lim_{\epsilon\to0}\bigl(\mathcal{J}^{Ar,\mu''^*\overline{K_{X''(k)}+\epsilon L''(k)}}(X''(k+1),L''(k+1),h(k))\\
    &-\mathcal{J}^{Ar,\overline{K_{X''(k)}+\epsilon L''(k)}}(X''(k),L''(k),h(k))\bigr)\\
    &=\mathcal{J}^{Ar,\mu''^*\overline{K_{X''(k)}}}(X''(k+1),L''(k+1),h_L(k))-\mathcal{J}^{Ar,\overline{K_{X''(k)}}}(X''(k),L''(k),h_L(k)).
\end{align*}
By \eqref{2} and the above equation, we have that
\[
\mathcal{J}^{Ar,\mu''^*\overline{K_{X''(k)}}}(X''(k+1),L''(k+1),h_L(k))-\mathcal{J}^{Ar,\overline{K_{X''(k)}}}(X''(k),L''(k),h_L(k))\ge0.
\]
Since $\overline{X_{\mathfrak{p}_k}}$ is locally F-pure, so is $X_{\mathfrak{p}_k}$.
By the same argument of the proof of Case 2 in \cite[Theorem 3.20]{H22} and the local F-purity of $X_{\mathfrak{p}_k}$, we have that 
\begin{align*}
&h_K(X''(k+1),L''(k+1),h(k))-h_K(X''(k),L''(k),h(k))\\
\ge&\mathcal{J}^{Ar,\mu''^*\overline{K_{X''(k)}}}(X''(k+1),L''(k+1),h_L(k))-\mathcal{J}^{Ar,\overline{K_{X''(k)}}}(X''(k),L''(k),h_L(k))
\ge0.
\end{align*}
This shows that $\eqref{eq--keyk}$ also holds in this case.
We complete the proof of Theorem \ref{hK.min} by the argument of the third paragraph of this proof.
\end{proof}


\begin{ack}
This work is partially 
supported by JSPS KAKENHI 	22J20059  
(Grant-in-Aid for JSPS Fellows DC1) 
for M.H. 
and 
KAKENHI 16H06335 (Grant-in-Aid for Scientific Research (S)), 
20H00112 (Grant-in-Aid for Scientific Research (A)), 
21H00973 (Grant-in-Aid for Scientific Research (B)) 
and 
KAKENHI 18K13389 (Grant-in-Aid for Early-Career Scientists) for Y.O. We 
thank Shunsuke Takagi and 
Shou Yoshikawa for useful comments. 
We also thank Teppei Takamatsu and Swaraj Pande 
for allowing us to write 
Example \ref{ex} \eqref{ex4}. 
We would like to thank the referee for the  
careful reading and suggestions. 
\end{ack}


\vspace{3mm}
\noindent 
{\tt hattori.masafumi.47z@st.kyoto-u.ac.jp}, \\
{\tt yodaka@math.kyoto-u.ac.jp} \\ 
Department of Mathematics, Kyoto university, 
Kyoto 606-8502, Japan \\ 

\end{document}